\documentclass[11pt,a4paper]{amsart}
\usepackage{latexsym}
\usepackage{amssymb}
\usepackage[all]{xy}

\newtheorem{theorem}{Theorem}[section]
\newtheorem{corollary}[theorem]{Corollary}
\newtheorem{lemma}[theorem]{Lemma}
\newtheorem{proposition}[theorem]{Proposition}
\newtheorem{definition}[theorem]{Definition}
\newtheorem{example}[theorem]{Example}
\newtheorem{remark}[theorem]{Remark}

\newcommand{\Ker}{{\rm Ker}}
\newcommand{\Coker}{{\rm Coker}}

\begin{document}
\sloppy

\title{Maximal exact structures on additive categories revisited}

\author{Septimiu Crivei}

\address{Faculty of Mathematics and Computer Science \\ ``Babe\c s-Bolyai" University \\ Str. Mihail Kog\u alniceanu 1
\\ 400084 Cluj-Napoca, Romania} \email{crivei@math.ubbcluj.ro}

\subjclass[2000]{18E10, 18G50} \keywords{Additive category, exact category, weakly idempotent complete category,
pre-abelian category, stable exact sequence.}

\begin{abstract} Sieg and Wegner showed that the stable exact sequences define a maximal exact structure
(in the sense of Quillen) in any pre-abelian category \cite{SW}. We generalize this result for weakly idempotent
complete additive categories.
\end{abstract}

\thanks{This work was supported by the Romanian grant PN-II-ID-PCE-2008-2 project ID\_2271. I would like to thank
Silvana Bazzoni for inspiring discussions on exact categories and kind hospitality during my visit at the Department
of Mathematics of Universit\'a di Padova in February-March 2010. Also, I would like to thank Dennis Sieg and Sven-Ake
Wegner for providing their paper \cite{SW}.}

\date{December 30, 2010}

\maketitle            

\section{Introduction}

Several notions of exact categories have been defined in the literature, see Barr \cite{Barr}, Heller \cite{Heller},
Quillen \cite{Q} or Yoneda \cite{Y}. They provide a suitable setting for developping a relative homological algebra, and
have important applications in different fields such as algebraic geometry, algebraic and functional analysis, algebraic
$K$-theory etc. (e.g., see \cite{Buhler} for further details).

We shall consider here the concept of exact additive category defined by Quillen \cite{Q} and refined by
Keller \cite{Keller}. In any additive category, the class of all split exact sequences defines an exact structure, and
this is the smallest exact structure. On the other hand, the other extreme, namely the class of all kernel-cokernel
pairs, defines an exact structure provided the category is quasi-abelian \cite{Rump01}, but fails to define an exact
structure in arbitrary additive categories (see the example in \cite{Rump08}).

Recently, Sieg and Wegner \cite{SW} showed that the stable exact sequences in the sense of \cite{RW} define a maximal
exact structure in any pre-abelian category, i.e. an additive category with kernels and cokernels. We shall generalize
this result to weakly idempotent complete additive categories, i.e. additive categories in which every section has a
cokernel, or equivalently, every retraction has a kernel (e.g., see \cite{Buhler}). Clearly, every pre-abelian category
is weakly idempotent complete additive. But there are significant examples of weakly idempotent complete categories
which are not pre-abelian. For instance, using the terminology from \cite{Prest}, any finitely accessible additive
category which is not locally finitely presented is weakly idempotent complete, but not pre-abelian (see Example
\ref{e:ex} below). Let us also point out that the assumption on the additive category to be weakly idempotent complete
is rather mild. This is because every additive category has an idempotent-splitting completion, also called Karoubian
completion (see \cite[p.~75]{K}), which in turn is weakly idempotent complete \cite{Buhler}. 

\section{Preliminaries}

Throughout the paper we shall use the setting of an additive category $\mathcal{C}$. In this section we give examples
of weakly idempotent complete additive categories which are not pre-abelian, and we introduce the needed
terminology.

\subsection{Examples} Following the terminology from \cite{Prest}, an additive category $\mathcal{C}$ is called
\emph{finitely accessible} if it has direct limits, the class of finitely presented objects is skeletally small, and
every object is a direct limit of finitely presented objects. Also, $\mathcal{C}$ is called \emph{locally finitely
presented} if it is finitely accessible and cocomplete (i.e., it has all colimits), or equivalently, it is finitely
accessible and complete (i.e., it has all limits). 

\begin{example} \label{e:ex} \rm (1) Let $\mathcal{C}$ be a finitely accessible additive category which is not locally
finitely presented. For instance, take the category of flat right modules over a ring which is not left coherent (see
\cite{Prest}). Then $\mathcal{C}$ is weakly idempotent complete, but not pre-abelian. Indeed, since $\mathcal{C}$
is finitely accessible, it has split idempotents \cite[2.4]{AR}, and so it is weakly idempotent complete \cite{Buhler}.
On the other hand, a finitely accessible category is locally finitely presented if and only if it has cokernels
\cite[Corollary~3.7]{Prest}. Hence $\mathcal{C}$ is not pre-abelian.  

(2) Any triangulated category is weakly idempotent complete additive, and its maximal exact structure is the trivial
one. Hence it is pre-abelian if and only if it is semi-simple (in the sense that every morphism factors into a
retraction followed by a section). 

(3) Any non-abelian category of finitely presented modules over a ring is weakly idempotent complete additive, but not
pre-abelian. It has cokernels, hence split idempotents, but not enough kernels.
\end{example}

\subsection{Pullbacks} We shall need the following two results on pullbacks, whose duals for pushouts hold as well.

\begin{lemma}{\cite[Lemma~5.1]{Kelly}} \label{l:PB} Consider the following diagram in which the squares are commutative
and the right square is a
pullback: 
\[\SelectTips{cm}{}
\xymatrix{
A' \ar[d]_f \ar[r]^{i'} & B' \ar[d]_g  \ar[r]^{d'} & C' \ar[d]^h \\ 
A \ar[r]_i & B\ar[r]_d & C 
}\]
Then the left square is a pullback if and only if so is the rectangle.
\end{lemma}

\begin{lemma}{\cite[Theorem~5]{RW}} \label{l:RW} Let $d:B\to C$ and $h:C'\to C$ be morphisms such that $d$ has a kernel
$i:A\to B$, and the pullback of $d$ and $h$ exists. Then there is a commutative diagram 
\[\SelectTips{cm}{}
\xymatrix{
A \ar@{=}[d] \ar[r]^{i'} & B' \ar[d]_g \ar[r]^{d'} & C' \ar[d]^h \\ 
A \ar[r]^i & B \ar[r]^d & C  
}\] 
in which the right square is a pullback and $i':A\to B'$ is the kernel of $d'$.
\end{lemma}

\subsection{Stable exact sequences}

The following special kernels and cokernels will be of fundamental importance for our topic. We extend their definition
from the setting of pre-abelian categories, as given in \cite{RW}, to arbitrary additive categories. 

\begin{definition} \rm A cokernel $d:B\to C$ is called a \emph{semi-stable cokernel} if the pullback of $d$ along an
arbitrary morphism $h:C'\to C$ exists and is again a cokernel, i.e. there is a pullback square
\[\SelectTips{cm}{}
\xymatrix{
B' \ar[d]_g \ar[r]^{d'} & C' \ar[d]^h \\ 
B \ar[r]^{d} & C  
}\] 
with the morphism $d':B'\to C'$ a cokernel. The notion of \emph{semi-stable kernel} is defined dually. 

A short exact sequence, i.e. a kernel-cokernel pair, $A\stackrel{i}\to B\stackrel{d}\to C$ is called \emph{stable} if
$i$ is a semi-stable kernel and $d$ is a semi-stable cokernel. 
\end{definition}

Let us note some useful remarks, whose dual versions hold as well.

\begin{remark} \label{r:rem} \rm (i) Every semi-stable cokernel $d:B\to C$ has a kernel (namely, its pullback along
the morphism $0\to C$). Hence every semi-stable cokernel is the cokernel of its kernel (e.g., by the
dual of \cite[Chapter~IV, Proposition~2.4]{Ste}, whose proof works in arbitrary additive categories). 

(ii) The pullback of a semi-stable cokernel along an arbitrary morphism exists and is again a semi-stable cokernel by
Lemma \ref{l:PB}.

(iii) Every isomorphism is a semi-stable cokernel.
\end{remark}

\subsection{Exact categories} We shall consider the following concept of exact category given by Quillen \cite{Q} and
refined by Keller \cite{Keller}.

\begin{definition} \rm By an \emph{exact category} we mean an additive category $\mathcal{C}$ endowed with a
distinguished class $\mathcal{E}$ of short exact sequences satisfying the axioms $[E0]$, $[E1]$, $[E2]$ and $[E2^{\rm
op}]$ below. The short exact sequences in $\mathcal{E}$ are called \emph{conflations}, whereas the kernels and cokernels
appearing in such exact sequences are called \emph{inflations} and \emph{deflations} respectively. 

$[E0]$ The identity morphism $1_0:0\to 0$ is a deflation.

$[E1]$ The composition of two deflations is again a deflation.

$[E2]$ The pullback of a deflation along an arbitrary morphism exists and is again a deflation.

$[E2^{\rm op}]$ The pushout of an inflation along an arbitrary morphism exists and is again an inflation.
\end{definition}

Note that the duals of the axioms $[E0]$ and $[E1]$ hold as well (see \cite{Keller}). Some examples of exact categories
are the following. 

\begin{example} \rm (1) It is well-known that in any additive category the split short exact sequences define an exact
structure, and this is the minimal one. 

(2) Recall that an additive category is called \emph{quasi-abelian} if it is \emph{pre-abelian} (i.e. it has kernels and
cokernels), the pushout of any kernel along an arbitrary morphism is a kernel, and the pullback of any cokernel along an
arbitrary morphism is a cokernel. In any quasi-abelian category the short exact sequences define an exact
structure, and this is the maximal one \cite{Rump01}.

(3) In any pre-abelian category the stable exact sequences define an exact structure, and this is the maximal one
\cite{SW}.
\end{example}

\section{The maximal exact structure}

In this section we shall extend the main result of \cite{SW} from pre-abelian categories to weakly idempotent complete
additive categories. We shall state and prove some essential results on semi-stable cokernels. Note that their
dual versions for semi-stable kernels hold as well. The setting will be that of an additive category $\mathcal{C}$, if
not specified otherwise.

The following result is modelled after \cite[Theorem~2]{RW}. We include a proof for completeness.

\begin{proposition} \label{p:E1} The composition of two semi-stable cokernels is a semi-stable cokernel.
\end{proposition}

\begin{proof} Let $d:B\to C$ and $p:C\to D$ be semi-stable cokernels. Then $d=\Coker(i)$ and
$p=\Coker(h)$, where $i=\Ker(d):A\to B$ and $h=\Ker(p):C'\to C$. Forming the pullback of $d$
and $h$, by Lemma \ref{l:RW} we have the following diagram with commutative squares:
\[\SelectTips{cm}{}
\xymatrix{
A \ar@{=}[d] \ar[r]^{i'} & B' \ar[d]_g \ar[r]^{d'} & C' \ar[d]^h \\ 
A \ar[r]^i & B \ar[d]_{pd} \ar[r]^d & C \ar[d]^p \\
 & D\ar@{=}[r] & D  
}\] 
in which $d'$ is a cokernel. 

We claim that $pd=\Coker(g)$. Let $u:B\to E$ be a morphism such that $ug=0$. Since $ui=0$ and $d=\Coker(i)$, there is a
unique morphism $v:C\to E$ such that $vd=u$. Since $vhd'=0$ and $d'$ is an epimorphism, we have $vh=0$. But
$p=\Coker(h)$, and so there is a unique morphism $w:D\to E$ such that $wp=v$. Now we have $wpd=u$. The fact that $p$ is
an epimorphism ensures the uniqueness of such a morphism $w$ with $wp=v$. Therefore, $pd=\Coker(g)$. 

In order to get the pullback of an arbitrary morphism $k:F\to D$ and $pd:B\to D$, construct the pullback of $k$ and
$p:C\to D$, and then the pullback of the resulting morphism and $d:B\to C$. Both of them yield semi-stable cokernels by
Remark \ref{r:rem}. Now the pullback of $pd$ along an arbitrary morphism exists by Lemma \ref{l:PB}, and it is the
resulting rectangle. Moreover, by the first part of the proof, it is a cokernel as the composition of two semi-stable
cokernels. 
\end{proof}

\begin{lemma} \label{l:ds} The direct sum of two semi-stable cokernels is a semi-stable cokernel.
\end{lemma}

\begin{proof} Let $d:B\to C$ and $d':B'\to C'$ be semi-stable cokernels. Consider the pullback square
\[\SelectTips{cm}{}
\xymatrix{
B\oplus B' \ar[d]_{\left [\begin{smallmatrix} 1&0 \end{smallmatrix}\right ]} \ar[r]^{\left [\begin{smallmatrix}
d&0\\0&1 \end{smallmatrix}\right ]} & C\oplus B' \ar[d]^{\left [\begin{smallmatrix} 1&0 \end{smallmatrix}\right ]} \\ 
B \ar[r]^d & C  
}\] 
Then $\left [\begin{smallmatrix} d&0\\0&1 \end{smallmatrix}\right ]:B\oplus B'\to C\oplus B'$ is a semi-stable cokernel
by Remark \ref{r:rem}. Similarly, $\left [\begin{smallmatrix} 1&0\\0&d' \end{smallmatrix}\right ]:C\oplus B'\to C\oplus
C'$ is a semi-stable cokernel. Therefore, their composition $\left [\begin{smallmatrix} d&0\\0&d'
\end{smallmatrix}\right ]:B\oplus B'\to C\oplus C'$, that is $d\oplus d'$, is a semi-stable cokernel by Proposition
\ref{p:E1}. 
\end{proof}

\begin{corollary} \label{c:proj} Every projection onto a direct summand  is a semi-stable cokernel.
\end{corollary}

\begin{proof} Consider a projection $\left [\begin{smallmatrix} 0&1 \end{smallmatrix}\right ]:B\oplus
D\stackrel{}\longrightarrow D$. By a diagram as in the proof of Lemma \ref{l:ds} with $C=0$, it follows that there
exists the pullback of the cokernel $B\to 0$ and any morphism $B'\to 0$. Moreover,  the resulting morphism $\left
[\begin{smallmatrix} 0&1 \end{smallmatrix}\right ]:B\oplus B'\to B'$ is a cokernel, as the direct sum of the cokernels
$B\to 0$ and $1_{B'}$. Hence $B\to 0$ is a semi-stable cokernel. To conclude, note that the morphism $B\oplus
D\stackrel{\left [\begin{smallmatrix} 0&1 \end{smallmatrix}\right ]}\longrightarrow D$ is the direct sum of the
semi-stable cokernels $B\to 0$ and $1_D$, and use Lemma \ref{l:ds}.
\end{proof}

Recall that an additive category is called \emph{weakly idempotent complete} if every retraction has a kernel
(equivalently, every section has a cokernel) (e.g., see \cite{Buhler}). The next result will be a key step in the proof
of our main theorem. It generalizes \cite[Proposition~5.12]{Kelly} and \cite[Proposition~1.1.8]{Sch}.

\begin{proposition} \label{p:obscure} Let $\mathcal{C}$ be weakly idempotent complete. Let $d:B\to C$ and $p:C\to D$
be morphisms such that $pd:B\to D$ is a semi-stable cokernel. Then $p$ is a semi-stable cokernel.
\end{proposition}

\begin{proof} We first show that $p$ has a kernel. Since $pd$ is a semi-stable cokernel, there is a 
pullback square \[\SelectTips{cm}{}
\xymatrix{
L \ar[r]^t \ar[d]_q & C \ar[d]^p \\ 
B \ar[r]^{pd} & D
} \] 
The pullback property implies the existence of a morphism $r:B\to L$ such that $tr=d$ and $qr=1_B$. Hence $q$ is a
retraction, and thus, by assumption, it has a kernel $l:C'\to L$. Using again the pullback property, it follows easily
that $h=tl:C'\to C$ is the kernel of $p$. 

Now let $g:B'\to B$ be the kernel of $pd$. Since $pd$ is a semi-stable cokernel, we have $pd=\Coker(g)$. We obtain the
following commutative left diagram:
\[\SelectTips{cm}{}
\xymatrix{
B' \ar[r]^{d'} \ar[d]_g & C' \ar[d]^h \\ 
B \ar[r]^d \ar[d]_{pd}& C \ar[d]^p \\
D \ar@{=}[r] & D 
} 
\qquad \SelectTips{cm}{}
\xymatrix{
B\oplus C' \ar[r]^{\left [\begin{smallmatrix} d&h \end{smallmatrix}\right ]} \ar[d]_{\left
[\begin{smallmatrix} 1&0 \end{smallmatrix}\right ]} & C \ar[d]^p \\ 
B \ar[r]^{pd} & D }\] 

We claim that the right diagram is a pullback. To this end, let $\alpha:E\to C$ and $\beta:E\to B$ be morphisms
such that $p\alpha=pd\beta$. Since $p(d\beta-\alpha)=0$ and $h=\Ker(p)$, there is a unique morphism
$\delta:E\to C'$ such that $d\beta-\alpha=h\delta$. Then it is easy to check that $\left [\begin{smallmatrix} \beta \\
-\delta \end{smallmatrix}\right ]$ is the unique morphism $\left [\begin{smallmatrix} u \\ v
\end{smallmatrix}\right ]:E\to B\oplus C'$ such that $\left [\begin{smallmatrix} d & h \end{smallmatrix}\right
]\left [\begin{smallmatrix} u \\ v \end{smallmatrix}\right ]=\alpha$ and $\left [\begin{smallmatrix} 1 & 0
\end{smallmatrix}\right ]\left [\begin{smallmatrix} u \\ v \end{smallmatrix}\right ]=\beta$, and so the square is a
pullback. Now $\left [\begin{smallmatrix} d & h \end{smallmatrix}\right ]$ is a cokernel, because $pd$ is a semi-stable
cokernel.
 
We claim that $p=\Coker(h)$. Let $w:C\to F$ be a morphism such that $wh=0$. Since $wdg=0$ and $pd=\Coker(g)$, there is
a morphism $t:D\to F$ such that $tpd=wd$. It follows that $(tp-w)\left [\begin{smallmatrix} d & h
\end{smallmatrix}\right ]=0$, whence we have $tp=w$, because $\left [\begin{smallmatrix} d & h \end{smallmatrix}\right]$
is an epimorphism. Since $p$ is an epimorphism, we have the uniqueness of the morphism $t:D\to F$ such that $tp=w$.
Hence $p=\Coker(h)$.  

Now let $c:G\to D$ be a morphism. We shall show that there exists the pullback of $p$ and $c$. We may write $\left
[\begin{smallmatrix} p & 0 \end{smallmatrix}\right ]$ as the composition of the following morphisms:
\[\xymatrix{C\oplus B\ar[r]^{\left [\begin{smallmatrix} 1&-d\\0&1 \end{smallmatrix}\right ]} & C\oplus B\ar[r]^{\left
[\begin{smallmatrix} 1&0 \\ 0&pd \end{smallmatrix}\right ]} & C\oplus D\ar[r]^{\left
[\begin{smallmatrix} 1&0 \\ p&1 \end{smallmatrix}\right ]} & C\oplus D\ar[r]^{\left
[\begin{smallmatrix} 0&1 \end{smallmatrix}\right ]} & D}\] The first and the third morphisms are isomorphisms, and so
they are semi-stable cokernels. The second morphism is a semi-stable cokernel by Lemma
\ref{l:ds}. The last morphism is a semi-stable cokernel by Corollary \ref{c:proj}. Therefore, their composition $\left
[\begin{smallmatrix} p & 0 \end{smallmatrix}\right ]$ is also a semi-stable cokernel by Proposition \ref{p:E1}. Hence
$\left [\begin{smallmatrix} p & 0 \end{smallmatrix}\right ]$ and $c$ have a pullback square as follows:
\[\SelectTips{cm}{}
\xymatrix{
Y\ar[r]^{\gamma} \ar[d]_{\left [\begin{smallmatrix} \alpha' \\ \beta' \end{smallmatrix}\right ]} & G \ar[d]^c \\ 
C\oplus B \ar[r]^{\left [\begin{smallmatrix} p&0 \end{smallmatrix}\right ]} & D }
\] Consider the morphism $\left [\begin{smallmatrix} 0\\1 \end{smallmatrix}\right ]:B\to C\oplus B$. Since $\left
[\begin{smallmatrix} p&0 \end{smallmatrix}\right ]\left [\begin{smallmatrix} 0\\1 \end{smallmatrix}\right ]=0=c0$, by
the pullback property there is a unique morphism $\delta:B\to Y$ such that $\left [\begin{smallmatrix} \alpha' \\ \beta'
\end{smallmatrix}\right ]\delta=\left [\begin{smallmatrix} 0\\1 \end{smallmatrix}\right ]$ and $\gamma \delta=0$. In
particular, $\beta' \delta=1_B$, and so $\beta'$ is a retraction. Since $\mathcal{C}$ is weakly idempotent complete,
$\beta'$ has a kernel, say $i:K\to Y$. Let us show now that the following square
\[\SelectTips{cm}{}
\xymatrix{
K\ar[r]^{\gamma i} \ar[d]_{\alpha' i} & G \ar[d]^c \\ 
C \ar[r]^p & D }
\]
is a pullback of $p$ and $c$. To this end, let $a:E'\to C$ and $b:E'\to G$ be morphisms such that $pa=cb$. Then $\left
[\begin{smallmatrix} p&0 \end{smallmatrix}\right ]\left [\begin{smallmatrix} a\\0 \end{smallmatrix}\right ]=cb$, hence
the pullback square of $\left [\begin{smallmatrix} p&0 \end{smallmatrix}\right ]$ and $c$ implies the existence of a
unique morphism $v':E'\to Y$ such that $\left [\begin{smallmatrix} \alpha' \\ \beta' \end{smallmatrix}\right ]v'=\left
[\begin{smallmatrix} a\\0 \end{smallmatrix}\right ]$ and $\gamma v'=b$. Since $\beta' v'=0$ and $i=\Ker(\beta')$, there
is a unique morphism $w':E'\to K$ such that $v'=iw'$. Then we have $\alpha' iw'=\alpha' v'=a$ and $\gamma iw'=\gamma
v'=b$. Let us show that the morphism $w':E\to K$ is unique with these properties. Suppose that there is another morphism
$w'':E'\to K$ such that $\alpha' iw''=a$ and $\gamma iw''=b$. It follows that $\left [\begin{smallmatrix} \alpha' \\
\beta' \end{smallmatrix}\right ](iw'-iw'')=\left [\begin{smallmatrix} 0 \\ 0 \end{smallmatrix}\right ]$ and
$\gamma(iw'-iw'')=0$. Then we have $iw'-iw''=0$ by the pullback property of $\left [\begin{smallmatrix} p&0
\end{smallmatrix}\right ]$ and $c$, and so $w'=w''$, because $i$ is a monomorphism.

Now consider the pullback of $pd$ and $c$, say 
\[\SelectTips{cm}{}
\xymatrix{
K'\ar[r]^{\gamma'} \ar[d]_{\alpha''} & G \ar[d]^c \\ 
B \ar[r]^{pd} & D }
\] The pullback property of $p$ and $c$ implies the factorization of $\gamma'$ through $\gamma i$. Since $pd$ is
a semi-stable cokernel, so is $\gamma'$. Moreover, by Lemma \ref{l:RW} $\gamma i$ has a kernel, because $p$ has a
kernel. Then $\gamma i$ must be a cokernel by an argument similar to the first part of the proof. Hence $p$ is a
semi-stable cokernel.
\end{proof}

Now we are in a position to prove our main result, which generalizes \cite[Theorem~3.3]{SW}. Having prepared the
setting, we shall follow a similar path as in the cited result, slightly simplifying the proof of axiom $[E1]$.

\begin{theorem} Let $\mathcal{C}$ be a weakly idempotent complete additive category. Then the stable exact
sequences define an exact structure on $\mathcal{C}$. Moreover, this is the maximal exact structure on
$\mathcal{C}$. 
\end{theorem}

\begin{proof} $[E0]$ This is clear.

$[E2]$ Let $A\stackrel{i}\to B\stackrel{d}\to C$ be a stable exact sequence, and let
$h:C'\to C$ be a morphism. Since $d$ is a semi-stable cokernel, we may consider the pullback of $d$ and $h$, and by
Lemma \ref{l:RW} we have a commutative diagram
\[\SelectTips{cm}{}
\xymatrix{
A \ar@{=}[d] \ar[r]^{i'} & B' \ar[d]_g \ar[r]^{d'} & C' \ar[d]^h \\ 
A \ar[r]^i & B \ar[r]^d  & C 
}\] 
in which $i'=\Ker(d')$ and $d'$ is a semi-stable cokernel, and so the upper row is a short exact sequence. Since $i=gi'$
is a semi-stable kernel, $i'$ is a semi-stable kernel by the dual of Proposition \ref{p:obscure}. 

$[E2^{\rm op}]$ Dual to $[E2]$.

$[E1]$ Let $A\stackrel{i}\to B\stackrel{d}\to C$ and $A'\stackrel{i'}\to C\stackrel{d'}\to D$ be stable exact
sequences. By Proposition \ref{p:E1}, $d'd:B\to D$ is a semi-stable cokernel. We shall show that its
kernel, say $j:K\to B$, is semi-stable. Since $d'dj=0$, there is a unique morphism $p:K\to
A'$ such that $dj=i'p$. We also have $d'd=\Coker(j)$. Note that we have the following equality:
\[\left [\begin{smallmatrix} d \\ 1 \end{smallmatrix}\right ]j=\left [\begin{smallmatrix} i'&0 \\ 0&1
\end{smallmatrix}\right ]\left [\begin{smallmatrix} p \\ j \end{smallmatrix}\right ]\] The morphism $\left
[\begin{smallmatrix} i'&0 \\ 0&1 \end{smallmatrix}\right ]:A'\oplus B\to C\oplus B$ is a semi-stable kernel by
Lemma \ref{l:ds}. Consequently, if we prove that $\left [\begin{smallmatrix} p \\ j \end{smallmatrix}\right ]:K\to
A'\oplus B$ is also a semi-stable kernel, then $j$ will be a semi-stable kernel by the duals of Propositions \ref{p:E1}
and \ref{p:obscure}, and we are done.

We claim first that there is a commutative diagram:
\[\SelectTips{cm}{}
\xymatrix{
A \ar@{=}[d] \ar[r]^k & K \ar[d]_j \ar[r]^p & A' \ar[d]^{i'} \\ 
A \ar[r]^i & B \ar[r]^d & C 
}\] in which the first row is a stable exact sequence. We shall show that the right square is a pullback. To this end,
let $\alpha:E\to A'$ and $\beta:E\to B$ be morphisms such that $i'\alpha=d\beta$. Since $d'd\beta=0$ and $j=\Ker(d'd)$,
there is a unique morphism $\gamma:E\to K$ such that $j\gamma=\beta$. We have $i'p\gamma=dj\gamma=d\beta=i'\alpha$,
whence $p\gamma=\alpha$, because $i'$ is a monomorphism. Moreover, it is easy to see that $\gamma$ is the unique
morphism with the required properties of the pullback. Now the existence of the required commutative diagram follows by
Lemma \ref{l:RW} and $[E2]$.

Next let us show that the following commutative square is a pushout:
\[\SelectTips{cm}{}
\xymatrix{
A \ar[d]_k \ar[r]^i & B \ar[d]^{\left [\begin{smallmatrix} 0 \\ 1 \end{smallmatrix}\right ]} \\ 
K \ar[r]^{\left [\begin{smallmatrix} p \\ j \end{smallmatrix}\right ]} & A'\oplus B 
}\] To this end, let $\alpha':K\to F$ and $\beta':B\to F$ be such that $\alpha' k=\beta' i$. Since $(\alpha'-\beta'
j)k=0$ and $p=\Coker(k)$, there is a unique morphism $\delta:A'\to F$ such that $\delta p=\alpha'-\beta'j$. Then it
follows that $\left [\begin{smallmatrix} \delta & \beta' \end{smallmatrix}\right ]$ is the unique morphism $\left
[\begin{smallmatrix} u & v \end{smallmatrix}\right ]:A'\oplus B\to F$ such that $\left [\begin{smallmatrix} u & v
\end{smallmatrix}\right ]\left [\begin{smallmatrix} p \\ j \end{smallmatrix}\right ]=\alpha'$ and $\left
[\begin{smallmatrix} u & v \end{smallmatrix}\right ]\left [\begin{smallmatrix} 0 \\ 1 \end{smallmatrix}\right ]=\beta'$.
Hence the square is a pushout.

Now $\left [\begin{smallmatrix} p \\ j \end{smallmatrix}\right ]$ is a semi-stable kernel by Remark
\ref{r:rem}. Consequently, $j$ is a semi-stable kernel by the above considerations.

Finally, consider an arbitrary exact structure $\mathcal{E}$ on $\mathcal{C}$, and let $X\stackrel{f}\to
Y\stackrel{g}\to Z$ be a conflation. Then by the axiom $[E2]$ for $\mathcal{E}$ the pullback of $g$
along an arbitrary morphism exists and is again a deflation, and so a cokernel. Hence $g$ is a semi-stable cokernel.
Dually, $f$ is a semi-stable kernel. Consequently, every conflation is a stable exact sequence. 
\end{proof}


\begin{thebibliography}{10}

\bibitem{AR} J. Ad\'amek and J. Rosick\'y, \emph{Locally presentable and accessible categories}, London Math. Soc.
Lecture Notes Series {\bf 189}, Cambridge University Press, Cambridge, 1994.

\bibitem{Barr} M. Barr, \emph{Exact categories}, Lect. Notes in Math. {\bf 236}, Springer, 1973.

\bibitem{Buhler} T. B\"uhler, \emph{Exact categories}, Expo. Math. {\bf 28} (2010), 1--69.

\bibitem{Heller} A. Heller, \emph{Homological algebra in abelian categories}, Ann. Math. {\bf 68} (1958), 484--525.

\bibitem{K} M. Karoubi, \emph{Alg\`ebres de Clifford et $K$-th\'eorie}, Ann. Sci. \'Ecole Norm. Sup. (4) {\bf 1}
(1968), 161--270.

\bibitem{Keller} B. Keller, \emph{Chain complexes and stable categories}, Manuscripta Math. {\bf 67} (1990), 379--417.

\bibitem{Kelly} G.M. Kelly, \emph{Monomorphisms, epimorphisms and pull-backs}, J. Austral. Math. Soc. {\bf 9} (1969),
124--142.

\bibitem{Prest} M. Prest, \emph{Definable additive categories: purity and model theory}, Mem. Amer. Math. Soc. {\bf
210} (2011), no. 987. 

\bibitem{Q} D. Quillen, \emph{Higher algebraic $K$-theory. I}, Algebraic $K$-theory, I: Higher $K$-theories (Proc.
Conf., Battelle Memorial Inst., Seattle, Wash., 1972), Springer, Berlin, 1973, pp. 85--147. Lecture Notes in Math.
{\bf 341}.

\bibitem{RW}  F. Richman and E.A. Walker, \emph{Ext in pre-Abelian categories}, Pacific J. Math. {\bf 71} (1977), 
521–-535. 

\bibitem{Rump01} W. Rump, \emph{Almost abelian categories}, Cah. Topologie G\'eom. Diff\'er. Cat\'egoriques {\bf 42}
(2001), 163--225.

\bibitem{Rump08} W. Rump, \emph{A counterexample to Raikov's conjecture}, Bull. London Math. Soc. {\bf 40} (2008), 
985–-994.

\bibitem{Sch} J.P. Schneiders, \emph{Quasi-abelian categories and sheaves}, M\'em. Soc. Math. Fr. (N.S.) {\bf 76}
(1999), pp. 1--140.

\bibitem{SW} D. Sieg and S.-A. Wegner, \emph{Maximal exact structures on additive categories}, Math. Nachr., 2010, to
appear.

\bibitem{Ste} B. Stenstr\"om, \emph{Rings of quotients}, Springer, Berlin, Heidelberg, New York, 1975.

\bibitem{Y} N. Yoneda, \emph{On Ext and exact sequences}, J. Fac. Sci. Univ. Tokyo Sect. I {\bf 8} (1960), 507--576.

\end{thebibliography}
\end{document}